\documentclass[11pt,a4paper]{article}

\pdfoutput=1 % for arxiv
\usepackage{epsf,epsfig,amsfonts,amsgen,amsmath,amstext,amsbsy,amsopn,amsthm,cases,listings,color
}
\usepackage{ebezier,eepic}
\usepackage{color}
\usepackage{multirow}
\usepackage{epstopdf}
\usepackage{graphicx}
\usepackage{pgf,tikz}
\usepackage{mathrsfs}
\usepackage[marginal]{footmisc}
\usepackage{enumitem}
\usepackage[titletoc]{appendix}
\usepackage{booktabs}
\usepackage{url}
\usepackage{mathtools}

\usepackage{pgfplots}
\usepackage{authblk}
\usepackage{amssymb}

\usepackage{wasysym}

\usepackage{empheq}

\usepackage{dsfont}
\usepackage{subfigure}
\usepackage{mathrsfs}
\usepackage{wasysym} 
\usetikzlibrary{arrows}
\usepackage{aligned-overset}

\usepackage{algorithm}
\usepackage{algpseudocode}

\usepackage{mathtools}

\usepackage{algcompatible}
\usepackage[backref=page]{hyperref}
\usepackage{bm}

\allowdisplaybreaks[1]

\usepackage{lineno}

\definecolor{uuuuuu}{rgb}{0.27,0.27,0.27}
\definecolor{sqsqsq}{rgb}{0.1255,0.1255,0.1255}

\newtheorem{definition}{Definition} [section]
\newtheorem{theorem}[definition]{Theorem}
\newtheorem{lemma}[definition]{Lemma}

\newtheorem{corollary}[definition]{Corollary}

\newtheorem{claim}[definition]{Claim}
\newtheorem{problem}[definition]{Problem}

\newenvironment{pf}{\noindent{\bf Proof}}

\begin{document}
%\pagewiselinenumbers% 按页重新编号
%\switchlinenumbers
\title{\bf\Large Hypergraph anti-Ramsey theorems}

\date{}
%%%%%%%%%%%%%%%%%%%%%%%%%%%%%%%%%%%%%%%%%%%%%%%%%%%%
\author[1]{Xizhi Liu\thanks{Research was supported by ERC Advanced Grant 101020255 and Leverhulme Research Project Grant RPG-2018-424. Email: \texttt{xizhi.liu.ac@gmail.com}}}
\author[2]{Jialei Song\thanks{Research was supported by Science and Technology Commission of Shanghai Municipality (No. 22DZ2229014) and National Natural Science Foundation of China (No. 12331014). Email: \texttt{jlsong@math.ecnu.edu.cn}}}

%%%%%%%%%%%%%%%%%%%%%%%%%%%%%%%%%%%%%%%%%%%%%%%%%%%%%
\affil[1]{Mathematics Institute and DIMAP,
            University of Warwick, 
            Coventry, CV4 7AL, UK}
\affil[2]{
School of Mathematical Sciences,  Key Laboratory of MEA (Ministry of Education) \& Shanghai Key Laboratory of PMMP,  East China Normal University, Shanghai 200241, China
}
%%%%%%%%%%%%%%%%%%%%%%%%%%%%%%%%%%%%%%%%%%%%%%%%%%%
\maketitle
%\footnote{footnote}
%%%%%%%%%%%%%%%%%%%%%%%%%%%%%%%%%%%%%%%%%%%%%%%%%
\begin{abstract}
The anti-Ramsey number $\mathrm{ar}(n,F)$ of an $r$-graph $F$ is 
the minimum number of colors needed to color the complete $n$-vertex $r$-graph to ensure the existence of a rainbow copy of $F$. 
We establish a removal-type result for the anti-Ramsey problem of $F$ when $F$ is the expansion of a hypergraph with a smaller uniformity.
We present two applications of this result. 
First, we refine the general bound $\mathrm{ar}(n,F) = \mathrm{ex}(n,F_{-}) + o(n^r)$ proved by Erd{\H o}s--Simonovits--S{\' o}s, where $F_{-}$ denotes the family of $r$-graphs obtained from $F$ by removing one edge. 
Second, we determine the exact value of $\mathrm{ar}(n,F)$ for large $n$ in cases where $F$ is the expansion of a specific class of graphs. This extends results of Erd{\H o}s--Simonovits--S{\' o}s on complete graphs 
to the realm of hypergraphs.

\medskip


\end{abstract}
%%%%%%%%%%%%%%%%%%%%%%%%%%%%%%%%%%%%%%%%%%%%%%%%%
\section{Introduction}\label{SEC:Introduction}
For a fixed integer $r\ge 2$, an $r$-graph $\mathcal{H}$ is a collection of $r$-subsets of some finite set $V$. We identify a hypergraph $\mathcal{H}$ with its edge set and use $V(\mathcal{H})$ to denote its vertex set.  
The \textbf{link} of a vertex $v\in V(\mathcal{H})$ in $\mathcal{H}$ is defined as 
\begin{align*}
    L_{\mathcal{H}}(v) := \left\{e\in \binom{V(\mathcal{H})}{r-1} \colon e\cup \{v\}\in \mathcal{H}\right\}. 
\end{align*}
We let $d_{\mathcal{H}}(v) := |L_{\mathcal{H}}(v)|$ denote the \textbf{degree} of $v$ in $\mathcal{H}$. 
Given a set $U \subseteq V(\mathcal{H})$, we use $\mathcal{H}[U]$ to denote the \textbf{induced subgraph} of $\mathcal{H}$ on $U$. 
While for $\ell\in \mathbb{N}$ and disjoint sets $U_1, \ldots, U_{\ell} \subseteq V(\mathcal{H})$, we use $\mathcal{H}[U_1, \ldots, U_{\ell}]$ to denote the \textbf{induced $\ell$-partite subgraph} of $\mathcal{H}$ with parts $U_1, \ldots, U_{\ell}$. 
In other words, $\mathcal{H}[U_1, \ldots, U_{\ell}]$ consists of all edges in $\mathcal{H}$ that contain at most one vertex from each $U_i$. 

The \textbf{anti-Ramsey number} $\mathrm{ar}(n,F)$ of an $r$-graph $F$ is the minimum number $m$ such that any surjective map $\chi \colon K_n^r \to [m]$ contains a rainbow copy of $F$, i.e. a copy of $F$ in which every edge receives a unique value under $\chi$. 
Given a family $\mathcal{F}$ of $r$-graphs, we say $\mathcal{H}$ is \textbf{$\mathcal{F}$-free}
if it does not contain any member of $\mathcal{F}$ as a subgraph.
The \textbf{Tur\'{a}n number} $\mathrm{ex}(n,\mathcal{F})$ of $\mathcal{F}$ is the maximum
number of edges in an $\mathcal{F}$-free $r$-graph on $n$ vertices.
The study of $\mathrm{ex}(n,\mathcal{F})$ and its variant has been a central topic in Extremal graph and hypergraph theory since the seminal work of Tur\'{a}n~\cite{T41}.
We refer the reader to  surveys~\cite{Fu91,Caen94,Sid95,Kee11} for more related results. 

In~\cite{ESS75}, Erd\H{o}s--Simonovits--S\'{o}s initiated the study of anti-Ramsey problems and proved various results for graphs and hypergraphs. 
In particular, they proved that for every $r$-graph with $r\ge 2$, 
\begin{align}\label{equ:ESS-hypergraph}
    \mathrm{ar}(n,F)
    \le \mathrm{ex}(n,F_{-}) + o(n^r), 
\end{align}
where $F_{-}$ denotes the family of $r$-graphs obtained from $F$ by removing exactly one edge. 
For complete graphs, they proved that for fixed $\ell \ge 2$ and sufficiently large $n$,  
\begin{align}\label{equ:ESS-complete-graph}
    \mathrm{ar}(n,K_{\ell+1})
    = \mathrm{ex}(n,K_{\ell}) + 2. 
\end{align}
Later, Montellano-Ballesteros and Neumann-Lara~\cite{MN02} refined their result by showing that~\eqref{equ:ESS-complete-graph} holds for all integers $n \ge \ell$. 

The study of the value of $\mathrm{ar}(n,F)$ for graph $F$ has received a lot of attention and there has been substantial progress since the work of Erd\H{o}s--Simonovits--S\'{o}s. 
Taking a more comprehensive perspective, Jiang~\cite{Jiang02} showed that $\mathrm{ar}(n,F) \le \mathrm{ex}(n,F_{-}) +O(n)$ for subdivided\footnote{Here, `subdivided' means every edge in $F$ is incidence to a vertex of degree two.} graphs $F$.  For further results on various classes of graphs, we refer the reader to a survey~\cite{FMO10} by Fujita--Magnant--Ozeki for more details. 
In contrast, the value of $\mathrm{ar}(n,F)$ is only known for a few classes of hypergraphs such as matchings (see e.g.~\cite{OY13,FK19,Jin21,GLP23}), linear paths and cycles (see e.g.~\cite{GLS20,TLY22,TLG23}), and the augmentation of certain linear trees (see e.g.~\cite{LS23a}). 
In this note, we contribute to hypergraph anti-Ramsey theory by refining~\eqref{equ:ESS-hypergraph} and extending~\eqref{equ:ESS-complete-graph} to hypergraph expansions. 
%
% It is easy to observe from the definition that for every $r$-graph $F$, 
%     \begin{align*}
%         \mathrm{ex}\left(n, \{F\setminus e\colon e\in F\}\right) + 2 
%         \le \mathrm{ar}(n,F)
%         \le \mathrm{ex}(n, F)+1. 
%     \end{align*}
%

Let $r > k \ge 2$ be integers. 
The \textbf{expansion} $H_{F}^{r}$ of a $k$-graph $F$ is the $r$-graph obtained from $F$ by adding a set of $r-k$ new vertices to each edge,  ensuring that different edges receive disjoint sets (see Figure~\ref{fig:Path-Cycle-Expansion}).
The expansion $H_{\mathcal{F}}^{r}$ of a family $\mathcal{F}$ is simply the collection of expansions of all elements in $\mathcal{F}$. 
Expansions are important objects in Extremal set theory and Hypergraph Tur\'{a}n problems. 
Its was introduced by Mubayi in~\cite{M06} as a way to extend Tur\'{a}n's theorem to hypergraphs. 
There has been lots of progress in the study of expansion over the last few decades, and we refer the reader to the survey~\cite{MV15Survey} by Mubyai--Verstra\"{e}te for more related results.

%%%%%%%%%%%%%%%%%%%%%%%%%%%%%%%%%%%%%%%%%%%%%
%%%%%%%%%%%%%%%%%%%%%%%%%%%%%%%%%%%%%%
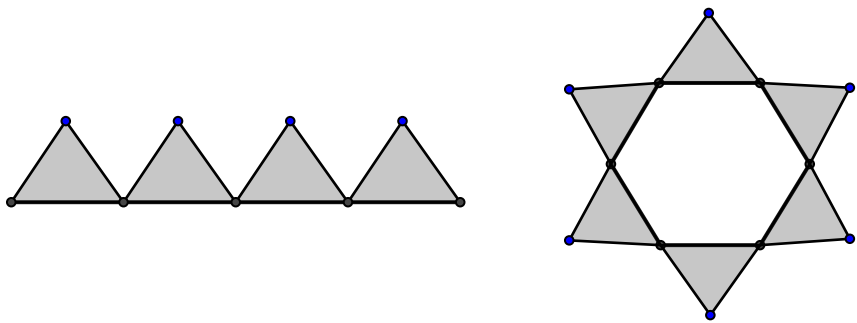
\begin{figure}[htbp]
\centering
\tikzset{every picture/.style={line width=0.75pt}} %set default line width to 0.75pt
\begin{tikzpicture}[x=0.75pt,y=0.75pt,yscale=-1,xscale=1,scale=0.8]

\draw[line width=1.5pt,color=sqsqsq]  (44,163) -- (114,163) -- (184,163) -- (254,163) -- (324,163) ;
%Shape: Triangle [id:dp11052726485885067] 
\draw[line width=1pt, fill=sqsqsq,fill opacity=0.25]   (78,112) -- (114,163) -- (44,163) -- cycle ;
%Shape: Triangle [id:dp07430588749907652] 
\draw[line width=1pt, fill=sqsqsq,fill opacity=0.25]  (148,112) -- (184,163) -- (114,163) -- cycle ;
%Shape: Triangle [id:dp49568886568617265] 
\draw[line width=1pt, fill=sqsqsq,fill opacity=0.25]   (218,112) -- (254,163) -- (184,163) -- cycle ;
%Shape: Triangle [id:dp004411457460550938] 
\draw[line width=1pt, fill=sqsqsq,fill opacity=0.25]  (288,112) -- (324,163) -- (254,163) -- cycle ;
\draw [fill=uuuuuu] (44,163) circle (2pt);
\draw [fill=uuuuuu]  (114,163) circle (2pt);
\draw [fill=uuuuuu]  (184,163) circle (2pt);
\draw [fill=uuuuuu]  (254,163) circle (2pt);
\draw [fill=uuuuuu]  (324,163)  circle (2pt);
\draw [fill=blue]   (78,112)  circle (2pt);
\draw [fill=blue]   (148,112) circle (2pt);
\draw [fill=blue]  (218,112)  circle (2pt);
\draw [fill=blue]   (288,112)  circle (2pt);

%Shape: Polygon [id:dp047441184112313683] 
\draw[line width=1.5pt,color=sqsqsq]  (542,139) -- (511,190) -- (449,190) -- (418,139) -- (448,88) -- (511,88) -- cycle ;

\draw [fill=uuuuuu] (542,139) circle (2pt);
\draw [fill=uuuuuu]  (511,190) circle (2pt);
\draw [fill=uuuuuu]  (449,190) circle (2pt);
\draw [fill=uuuuuu]  (418,139) circle (2pt);
\draw [fill=uuuuuu]  (448,88)  circle (2pt);
\draw [fill=uuuuuu] (511,88)  circle (2pt);

%Shape: Triangle [id:dp35102292243286315] 
\draw[line width=1pt, fill=sqsqsq,fill opacity=0.25]   (479,44) -- (511,88) -- (448,88) -- cycle ;
%Shape: Triangle [id:dp2468855030768613] 
\draw[line width=1pt, fill=sqsqsq,fill opacity=0.25]   (567,91) -- (542,139) -- (511,88) -- cycle ;
%Shape: Triangle [id:dp6235485588978371] 
\draw[line width=1pt, fill=sqsqsq,fill opacity=0.25]   (480,234) -- (449,190) -- (511,190) -- cycle ;
%Shape: Triangle [id:dp5280222854920176] 
\draw[line width=1pt, fill=sqsqsq,fill opacity=0.25]   (567,186) -- (511,190) -- (542,139) -- cycle ;
%Shape: Triangle [id:dp6818555953118328] 
\draw[line width=1pt, fill=sqsqsq,fill opacity=0.25]   (392,92) -- (448,88) -- (418,139) -- cycle ;
%Shape: Triangle [id:dp331025221334615] 
\draw[line width=1pt, fill=sqsqsq,fill opacity=0.25]   (392,187) -- (418,139) -- (449,190) -- cycle ;
\draw [fill=blue]  (479,44)  circle (2pt);
\draw [fill=blue]   (567,91)  circle (2pt);
\draw [fill=blue]  (480,234)   circle (2pt);
\draw [fill=blue]   (567,186)  circle (2pt);
\draw [fill=blue]  (392,92)  circle (2pt);
\draw [fill=blue]   (392,187)  circle (2pt);

\end{tikzpicture}
\caption{Expansions of $P_4$ and $C_6$ for $r=3$.}
\label{fig:Path-Cycle-Expansion}
\end{figure}
%uncomment if require: \path (0,254); %set diagram left start at 0, and has height of 254
%%%%%%%%%%%%%%%%%%%%%%%%%%%%%%%%%%%%%%
%%%%%%%%%%%%%%%%%%%%%%%%%%%%%%%%%%%%%%%%%%%%%

%%%%%%%%%%%%%%%%%%%%%%%%%%%%%%%%%%
\subsection{Expansion of hypergraphs}
In this subsection, we present a refinement of~\eqref{equ:ESS-hypergraph}. The following definitions will play a crucial role in our results.

Given a $k$-graph $F$, $k \ge 2$, and a vertex $u\in V(F)$, the \textbf{$u$-splitting} of $F$ is a $k$-graph, denoted by $F\vee u$, obtained from $F$ by 
\begin{itemize}
    \item removing the vertex $u$ and all edges containing $u$ from $F$, and then 
    \item adding a $d_{F}(u)$-set $\{v_e \colon e \in L_{F}(u)\}$ of new vertices to the vertex set and adding $\left\{e\cup \{v_e\} \colon e \in L_{F}(u)\right\}$ to the edge set. 
\end{itemize}
% \begin{enumerate}[label=(\roman*)]
%     \item removing vertex $u$ and all edges containing $u$ from $F$, and 
%     \item adding a $d_{F}(u)$-set $\{v_e \colon e \in L_{F}(u)\}$ of new vertices and adding $\left\{e\cup \{v_e\} \colon e \in L_{F}(u)\right\}$ to the edge set. 
% \end{enumerate}
In other words, 
\begin{align*}
    F\vee u
    := \left\{e\cup \{v_e\} \colon e \in L_{F}(u)\right\} \cup \left(F-u\right), 
\end{align*}
where $\{v_e \colon e \in L_{F}(u)\}$ is a $d_{F}(u)$-set of new vertices outside $V(F)$. 

Given an independent set $I:= \{u_1, \ldots, u_{\ell}\}$ in $F$, the  \textbf{$I$-splitting}  of $F$, denoted by $F\vee I$, is defined by recursively letting $F_0 := F$ and $F_i:= F_{i-1}\vee u_i$ for $i\in [\ell]$ (see Figure~\ref{fig:splitting}). 
It is easy to see that since $I$ is an independent set, the ordering of vertices in $I$ does not affect the resulting $k$-graph. 
The \textbf{splitting family} $\mathrm{Split}(F)$ of $F$ is defined as  
\begin{align*}
    \mathrm{Split}(F)
    := \left\{\hat{F} \colon \text{there exists an independent set $I$ in 
$F$ such that $\hat{F} = F\vee I$}\right\}. 
\end{align*}
In the definition above, we allow $I$ to be empty. 
Hence, $F\in \mathrm{Split}(F)$. 
Note that $|\hat{F}| = |F|$ for all $\hat{F} \in \mathrm{Split}(F)$. 

Given a coloring $\chi\colon K_{n}^{r} \to \mathbb{N}$, we say a subgraph $\mathcal{H} \subseteq K_{n}^{r}$ is \textbf{rainbow} if no two edges in $\mathcal{H}$ have the same color. 
The coloring $\chi$ is \textbf{rainbow-$\mathcal{F}$-free} for some family $\mathcal{F}$ if any rainbow subgraph of $\chi$-colored $K_{n}^r$ is $\mathcal{F}$-free. 

The main result of this subsection is as follows. 
\begin{theorem}\label{THM:antiRamsey-expansion-hypergraph-splitting}
    Let $r > k \ge 2$ be integers and $F$ be a $k$-graph. The following holds for all sufficiently large $n$.
    Suppose that $\chi\colon K_{n}^{r} \to \mathbb{N}$ is a rainbow-$H_{F}^{r}$-free coloring. 
    Then every rainbow subgraph $\mathcal{H} \subseteq K_{n}^{r}$ can be made $H_{F_{-}}^{r}$-free by removing at most $\left(|F|-1\right)\cdot \mathrm{ex}\left(n, \mathrm{Split}(F)\right)$ edges. 
    In particular, 
    \begin{align*}
        \mathrm{ar}(n, H_{F}^{r})
        \le \mathrm{ex}(n, H_{F_{-}}^{r}) + \left(|F|-1\right)\cdot \mathrm{ex}\left(n, \mathrm{Split}(F)\right) +1. 
    \end{align*}
\end{theorem}
Since $\mathrm{Split}(F)$ is a family of $k$-graphs (and $k \le r-1$), we have $\mathrm{ar}(n, H_{F}^{r}) \le \mathrm{ex}(n, H_{F_{-}}^{r}) + O(n^k)$, which improves the bound given by~\eqref{equ:ESS-hypergraph} for large $n$. 

Observe that if the graph $F$ is obtained from a bipartite graph by adding a forest to one part, then the family $\mathrm{Split}(F)$ contains a forest (obtained by splitting the other part in $F$). 
Therefore,  we have the following corollary. 
\begin{corollary}\label{CORO:antiRamsey-expansion-hypergraph-splittingk=2}
    Let $r\ge 3$ be integers and $F$ be a graph obtained from a bipartite graph by adding a forest to one part.  
    There exists a constant $C_F$ depending only on $F$ such that the following holds for every $n \in \mathbb{N}$. 
    Suppose that $\chi\colon K_{n}^{r} \to \mathbb{N}$ is a rainbow-$H_{F}^{r}$-free coloring. 
    Then every rainbow subgraph $\mathcal{H} \subseteq K_{n}^{r}$ can be made $H_{F_{-}}^{r}$-free by removing at most $C_F \cdot n$ edges. 
    In particular, 
    \begin{align*}
        \mathrm{ar}(n, H_{F}^{r})
        \le \mathrm{ex}(n, H_{F_{-}}^{r}) + C_F \cdot n. 
    \end{align*}
\end{corollary}
Notice that if a $k$-graph $F$ is $p$-partite, then $\mathrm{Split}(F)$ contains a $(p-1)$-partite $k$-graph (obtained by splitting an arbitrary part in $F$). 
Recall the (hypergraph) K\"{o}vari--S\'{o}s--Tur\'{a}n theorem~\cite{KST54,E64}, we obtain the following corollary.  
\begin{corollary}\label{CORO:antiRamsey-expansion-hypergraph-k+1-partite}
    Let $r> k \ge 2$ be integers and $F$ be a $(k+1)$-partite $k$-graph. 
    There exists constants $C_F, \alpha_F >0$ depending only on $F$ such that the following holds for every $n \in \mathbb{N}$. 
    Suppose that $\chi\colon K_{n}^{r} \to \mathbb{N}$ is a rainbow-$H_{F}^{r}$-free coloring. 
    Then every rainbow subgraph $\mathcal{H} \subseteq K_{n}^{r}$ can be made $H_{F_{-}}^{r}$-free by removing at most $C_F \cdot n^{k-\alpha_F}$ edges. 
    In particular, 
    \begin{align*}
        \mathrm{ar}(n,H_{F}^{r}) 
            \le \mathrm{ex}(n,H_{F_{-}}^{r}) + C_{F}\cdot n^{k-\alpha_F}.  
    \end{align*}
\end{corollary}
Theorem~\ref{THM:antiRamsey-expansion-hypergraph-splitting} will be proved in Section~\ref{SEC:proof-antiRamsey-expansion-hypergraph-splitting}. 

%%%%%%%%%%%%%%%%%%%%%%%%%%%%%%%%%
\subsection{Expansion of graphs}
In this subsection, we present an extension of~\eqref{equ:ESS-complete-graph} to the expansion of graphs. 
For convenience, we let $H_{\ell}^{r} := H_{K_{\ell}}^{r}$.  

Let $t$ be a positive integer. 
We use $F[t]$ to denote the \textbf{$t$-blowup} of $F$, i.e. $F[t]$ is the $r$-graph obtained from $F$ by replacing vertices with disjoint $t$-sets, and replacing each edge in $F$ by the corresponding complete $r$-partite $r$-graph. 
Given a family $\mathcal{F}$ of $r$-graphs, we let 
\begin{align*}
    \mathcal{F}[t] := \left\{F[t] \colon F\in \mathcal{F}\right\}. 
\end{align*}
Let $\ell \ge 2$ and $t \ge 4$ be integers. 
\begin{itemize}
    \item Let $K_{\ell}^{\alpha}[t]$ denote the graph obtained from the blowup $K_{\ell}[t]$  by adding a path of length two into one part (see Figure~\ref{fig:edge-critical}~(a)).
    \item Let $K_{\ell}^{\beta}[t]$ denote the graph obtained from the blowup $K_{\ell}[t]$  by adding two disjoint edges into one part (see Figure~\ref{fig:edge-critical}~(b)). 
    \item Let $K_{\ell}^{\gamma}[t]$ denote the graph obtained from the blowup $K_{\ell}[t]$ by adding two edges into two different parts (see Figure~\ref{fig:edge-critical}~(c)). 
\end{itemize}
The motivation for these definitions is that if $F$ is obtained from an edge-critical graph by adding one edge, then $F$ can be found within one of the previously defined graphs for sufficiently large $t$. 

%%%%%%%%%%%%%%%%%%%%%%%%%%%%
%%%%%%%%%%%%%%%%%%%%%%%%%%%%%%%%%%%%%%
\begin{figure}[htbp]
\centering
\tikzset{every picture/.style={line width=1pt}} %set default line width to 0.75pt        

\begin{tikzpicture}[x=0.75pt,y=0.75pt,yscale=-1,xscale=1,scale=0.9]

%Shape: Ellipse [id:dp37487356095039726] 
\draw   (65,164.82) .. controls (65,149.26) and (77.58,136.64) .. (93.09,136.64) .. controls (108.61,136.64) and (121.19,149.26) .. (121.19,164.82) .. controls (121.19,180.38) and (108.61,193) .. (93.09,193) .. controls (77.58,193) and (65,180.38) .. (65,164.82) -- cycle ;
%Shape: Ellipse [id:dp2811744344404663] 
\draw   (154.15,164.82) .. controls (154.15,149.26) and (166.73,136.64) .. (182.24,136.64) .. controls (197.76,136.64) and (210.33,149.26) .. (210.33,164.82) .. controls (210.33,180.38) and (197.76,193) .. (182.24,193) .. controls (166.73,193) and (154.15,180.38) .. (154.15,164.82) -- cycle ;
%Shape: Ellipse [id:dp7374644326789201] 
\draw   (109.57,88.18) .. controls (109.57,72.62) and (122.15,60) .. (137.67,60) .. controls (153.18,60) and (165.76,72.62) .. (165.76,88.18) .. controls (165.76,103.74) and (153.18,116.36) .. (137.67,116.36) .. controls (122.15,116.36) and (109.57,103.74) .. (109.57,88.18) -- cycle ;
%Straight Lines [id:da25180306624279414] 
\draw [color=blue]  (126.31,83.27) -- (145.53,90.73) ;
%Straight Lines [id:da7710773496545924] 
\draw  [color=blue]  (126.31,83.27) -- (145.13,76.73) ;
%Shape: Sawtooth Wave Form [id:dp106783262807155] 
\draw   (82.99,150.81) -- (122.46,93.76) -- (88.92,154.1) -- (128.39,97.06) -- (94.84,157.4) -- (134.31,100.35) -- (100.77,160.69) ;
%Shape: Sawtooth Wave Form [id:dp11238445617047033] 
\draw   (157.54,94.06) -- (189.25,153.36) -- (151.92,97.85) -- (183.63,157.15) -- (146.29,101.63) -- (178,160.93) -- (140.67,105.41) ;
%Shape: Sawtooth Wave Form [id:dp060209479206088545] 
\draw   (172.33,179.15) -- (105.32,178.95) -- (171.67,172.41) -- (104.66,172.2) -- (171,165.66) -- (103.99,165.46) -- (170.34,158.92) ;
%Shape: Ellipse [id:dp991553447653359] 
\draw   (255,164.82) .. controls (255,149.26) and (267.58,136.64) .. (283.09,136.64) .. controls (298.61,136.64) and (311.19,149.26) .. (311.19,164.82) .. controls (311.19,180.38) and (298.61,193) .. (283.09,193) .. controls (267.58,193) and (255,180.38) .. (255,164.82) -- cycle ;
%Shape: Ellipse [id:dp25328178134347157] 
\draw   (344.15,164.82) .. controls (344.15,149.26) and (356.73,136.64) .. (372.24,136.64) .. controls (387.76,136.64) and (400.33,149.26) .. (400.33,164.82) .. controls (400.33,180.38) and (387.76,193) .. (372.24,193) .. controls (356.73,193) and (344.15,180.38) .. (344.15,164.82) -- cycle ;
%Shape: Ellipse [id:dp07185765327537785] 
\draw   (299.57,88.18) .. controls (299.57,72.62) and (312.15,60) .. (327.67,60) .. controls (343.18,60) and (355.76,72.62) .. (355.76,88.18) .. controls (355.76,103.74) and (343.18,116.36) .. (327.67,116.36) .. controls (312.15,116.36) and (299.57,103.74) .. (299.57,88.18) -- cycle ;
%Straight Lines [id:da1567814059476078] 
\draw [color=blue]  (317,73) -- (338,73) ;
%Straight Lines [id:da6719522299656158] 
\draw  [color=blue]  (317,82.18) -- (338,82.18) ;
%Shape: Sawtooth Wave Form [id:dp5548617085352106] 
\draw   (272.99,150.81) -- (312.46,93.76) -- (278.92,154.1) -- (318.39,97.06) -- (284.84,157.4) -- (324.31,100.35) -- (290.77,160.69) ;
%Shape: Sawtooth Wave Form [id:dp5242910534691012] 
\draw   (347.54,94.06) -- (379.25,153.36) -- (341.92,97.85) -- (373.63,157.15) -- (336.29,101.63) -- (368,160.93) -- (330.67,105.41) ;
%Shape: Sawtooth Wave Form [id:dp019375091784886278] 
\draw   (362.33,179.15) -- (295.32,178.95) -- (361.67,172.41) -- (294.66,172.2) -- (361,165.66) -- (293.99,165.46) -- (360.34,158.92) ;
%Shape: Ellipse [id:dp8594868676634688] 
\draw   (443,165.82) .. controls (443,150.26) and (455.58,137.64) .. (471.09,137.64) .. controls (486.61,137.64) and (499.19,150.26) .. (499.19,165.82) .. controls (499.19,181.38) and (486.61,194) .. (471.09,194) .. controls (455.58,194) and (443,181.38) .. (443,165.82) -- cycle ;
%Shape: Ellipse [id:dp5949095442538193] 
\draw   (532.15,165.82) .. controls (532.15,150.26) and (544.73,137.64) .. (560.24,137.64) .. controls (575.76,137.64) and (588.33,150.26) .. (588.33,165.82) .. controls (588.33,181.38) and (575.76,194) .. (560.24,194) .. controls (544.73,194) and (532.15,181.38) .. (532.15,165.82) -- cycle ;
%Shape: Ellipse [id:dp19399481541256458] 
\draw   (487.57,89.18) .. controls (487.57,73.62) and (500.15,61) .. (515.67,61) .. controls (531.18,61) and (543.76,73.62) .. (543.76,89.18) .. controls (543.76,104.74) and (531.18,117.36) .. (515.67,117.36) .. controls (500.15,117.36) and (487.57,104.74) .. (487.57,89.18) -- cycle ;
%Shape: Sawtooth Wave Form [id:dp677522620191374] 
\draw   (460.99,151.81) -- (500.46,94.76) -- (466.92,155.1) -- (506.39,98.06) -- (472.84,158.4) -- (512.31,101.35) -- (478.77,161.69) ;
%Shape: Sawtooth Wave Form [id:dp5354105419796455] 
\draw   (535.54,95.06) -- (567.25,154.36) -- (529.92,98.85) -- (561.63,158.15) -- (524.29,102.63) -- (556,161.93) -- (518.67,106.41) ;
%Shape: Sawtooth Wave Form [id:dp8818336454523783] 
\draw   (550.33,180.15) -- (483.32,179.95) -- (549.67,173.41) -- (482.66,173.2) -- (549,166.66) -- (481.99,166.46) -- (548.34,159.92) ;
%Straight Lines [id:da964123955283273] 
\draw  [color=blue]  (505.33,81) -- (526.33,81) ;
%Straight Lines [id:da30452873258097135] 
\draw  [color=blue]  (458.33,166) -- (472.33,180) ;
%
% Text Node
\draw (125,210) node [anchor=north west][inner sep=0.75pt]   [align=left] {(a)};
% Text Node
\draw (320,210) node [anchor=north west][inner sep=0.75pt]   [align=left] {(b)};
% Text Node
\draw (510,210) node [anchor=north west][inner sep=0.75pt]   [align=left] {(c)};
\end{tikzpicture}
\caption{$K_{3}^{\alpha}[t]$, $K_{3}^{\beta}[t]$, and $K_{3}^{\gamma}[t]$.}
\label{fig:edge-critical}
\end{figure}
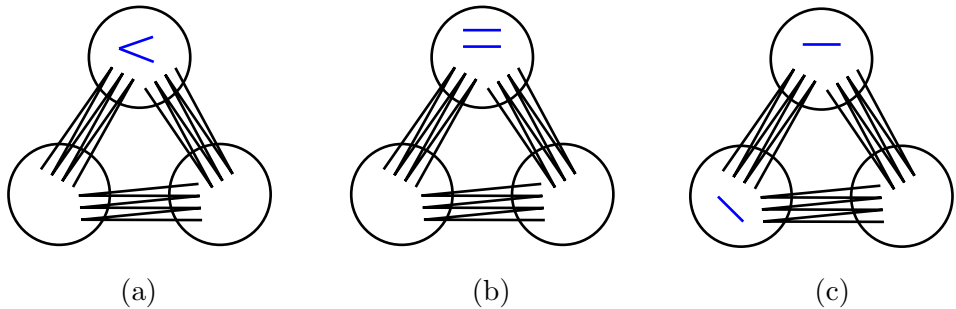
%uncomment if require: \path (0,254); %set diagram left start at 0, and has height of 254
%%%%%%%%%%%%%%%%%%%%%%%%%%%%%%%%%%%%%%
%%%%%%%%%%%%%%%%%%%%%%%%%%%%

Given integers $n \ge r \ge 2$.
Let $V_1 \sqcup \cdots \sqcup V_{\ell} = [n]$ be a partition such that $|V_{\ell}|+1\ge |V_1| \ge |V_2| \ge \cdots \ge |V_{\ell}|$. 
Let $T_{r}(n,\ell)$ denote the $r$-graph whose edge set consists of all $r$-subsets of $[n]$ that contain at most one vertex from each $V_i$.
Let $t_{r}(n,\ell)$ denote the number of edges in $T_{r}(n,\ell)$ and notice that $t_{r}(n,\ell) \sim \binom{\ell}{r}\left(\frac{n}{\ell}\right)^r$. 
Extending the classical Tur{\' a}n Theorem to hypergraphs, Mubayi proved in~\cite{M06} that $\mathrm{ex}(n,H_{\ell+1}^{r}) = (1+o(1))t_{r}(n,\ell)$ for all $\ell \ge r \ge 3$. 
Building on a stability theorem of Mubayi, Pikhurko~\cite{Pik13} later proved that for fixed $\ell \ge r \ge 3$,  $\mathrm{ex}(n,H_{\ell+1}^{r}) = t_{r}(n,\ell)$ holds for all sufficiently large $n$.

The main results in this subsection are as follows. 
\begin{theorem}\label{THM:antiRamsey-expansion-edge-criticalr=3}
    Let $t > \ell \ge 3$ be fixed integers. For all sufficiently large $n$, 
    \begin{align*}
        \mathrm{ar}(n,H_{F}^{3}) 
            = 
            \begin{cases}
                t_{3}(n,\ell) + \ell+1, & \quad\text{if}\quad F \in \left\{K_{\ell}^{\alpha}[t],\ K_{\ell}^{\beta}[t]\right\}, \\
                t_{3}(n,\ell) + 2, & \quad\text{if}\quad F = K_{\ell}^{\gamma}[t].
            \end{cases}
    \end{align*}
\end{theorem}
The situation for $r \ge 4$ is simpler. 
\begin{theorem}\label{THM:antiRamsey-expansion-edge-criticalr>3}
    Let $\ell \ge r \ge 4$ and $t \ge 4$ be fixed integers. 
    For all sufficiently large $n$, 
    \begin{align*}
        \mathrm{ar}(n,H_{F}^{r}) 
            = t_{r}(n,\ell) + 2
            \quad\text{for all}\quad 
            F\in \left\{K_{\ell}^{\alpha}[t],\ K_{\ell}^{\beta}[t],\ K_{\ell}^{\gamma}[t]\right\}. 
    \end{align*}
\end{theorem}
\textbf{Remarks.}
\begin{itemize}
    \item The analogous result for the case $r=2$ is handled by a result of Jiang--Pikhurko in~\cite{JP09}. 
    \item After submitting this manuscript, we learned that Li--Tang--Yan~\cite{LTY24a} independently obtained similar results.
\end{itemize}

Observe that for every $r\ge 3$, the $r$-graph $H_{\ell+1}^{r}$ is in the expansion of $K_{\ell}^{\gamma}[4]$. Hence, we obtain the following corollary, which is an extension of~\eqref{equ:ESS-complete-graph} to hypergraphs. 
\begin{corollary}
     Let $\ell \ge r \ge 3$ be fixed integers. 
    For all sufficiently large $n$, 
    \begin{align*}
        \mathrm{ar}(n,H_{\ell+1}^{r}) 
            = t_{r}(n,\ell) + 2.  
    \end{align*}
\end{corollary}
%
%\todo{$\ell+1$ should be $\ell+2$?}
Proofs for Theorems~\ref{THM:antiRamsey-expansion-edge-criticalr=3} and~\ref{THM:antiRamsey-expansion-edge-criticalr>3} are presented in Section~\ref{SEC:proof-antiRamsey-expansion-edge-criticalr=3}.

%\input{Preliminary}
% \section{Proofs}\label{SEC:proof}
% %
\section{Proof of Theorem~\ref{THM:antiRamsey-expansion-hypergraph-splitting}}\label{SEC:proof-antiRamsey-expansion-hypergraph-splitting}
%

%%%%%%%%%%%%%%%%%%%%%%%%%%%%%%%%
%%%%%%%%%%%%%%%%%%%%%%%%%%%%%%%%%%%%%%
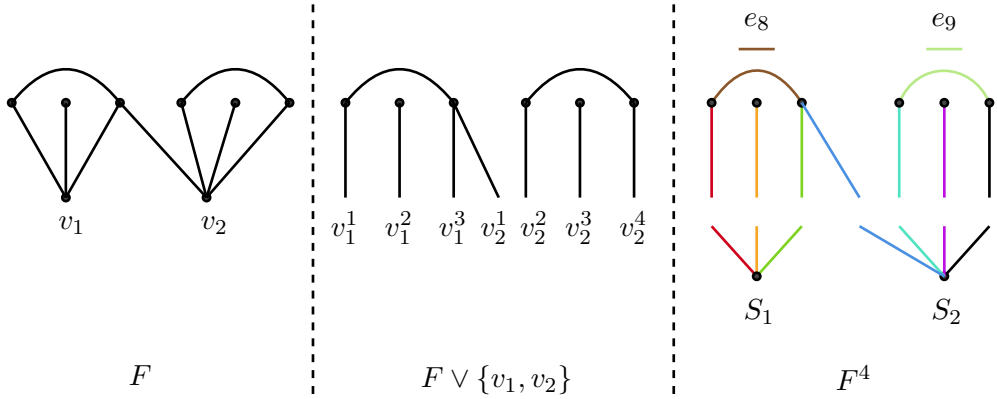
\begin{figure}[htbp]
\centering
\tikzset{every picture/.style={line width=1pt}} %set default line width to 0.75pt        

\begin{tikzpicture}[x=0.75pt,y=0.75pt,yscale=-1,xscale=1,scale=0.9]
%Curve Lines [id:da8381715939049139] 
\draw [color={rgb, 255:red, 139; green, 87; blue, 42 }  ,draw opacity=1 ]   (443,85) .. controls (462.33,56.67) and (480.33,65.67) .. (493,85) ;
%Straight Lines [id:da9092901767343571] 
\draw [color={rgb, 255:red, 139; green, 87; blue, 42 }  ,draw opacity=1 ]   (478,55) -- (458,55) ;
%Straight Lines [id:da2631582887037667] 
\draw [color={rgb, 255:red, 208; green, 2; blue, 27 }  ,draw opacity=1 ]   (443,85) -- (443,138) ;
%Straight Lines [id:da7551204032709633] 
\draw [color={rgb, 255:red, 245; green, 166; blue, 35 }  ,draw opacity=1 ]   (468,85) -- (468,138) ;
%Straight Lines [id:da6149357358212899] 
\draw [color={rgb, 255:red, 126; green, 211; blue, 33 }  ,draw opacity=1 ]   (493,85) -- (493,138) ;
\draw [fill=uuuuuu]   (443,85)  circle (1.5pt);
\draw [fill=uuuuuu]   (468,85)  circle (1.5pt);
\draw [fill=uuuuuu]   (493,85)  circle (1.5pt);
%
%Straight Lines [id:da7939838153885102] 
\draw [color={rgb, 255:red, 126; green, 211; blue, 33 }  ,draw opacity=1 ]   (468,182) -- (493,154) ;
%Straight Lines [id:da799912328673225] 
\draw [color={rgb, 255:red, 245; green, 166; blue, 35 }  ,draw opacity=1 ]   (468,182) -- (468,154) ;
%Straight Lines [id:da9417395620185047] 
\draw [color={rgb, 255:red, 208; green, 2; blue, 27 }  ,draw opacity=1 ]   (468,182) -- (443,154) ;
\draw [fill=uuuuuu]   (468,182)  circle (1.5pt);
%
%Curve Lines [id:da2682071365887633] 
\draw [color={rgb, 255:red, 184; green, 233; blue, 134 }  ,draw opacity=1 ]   (547,85) .. controls (557,60.67) and (587,61.67) .. (597,85)  ;
%Straight Lines [id:da7554963438269491] 
\draw [color={rgb, 255:red, 184; green, 233; blue, 134 }  ,draw opacity=1 ]   (562,55) -- (582,55) ;
%Straight Lines [id:da9507488736614491] 
\draw [color={rgb, 255:red, 189; green, 16; blue, 224 }  ,draw opacity=1 ]   (572,85) -- (572,138) ;
%Straight Lines [id:da9928688582852037] 
\draw [color={rgb, 255:red, 74; green, 144; blue, 226 }  ,draw opacity=1 ]   (493,85) -- (525,138) ;
%Straight Lines [id:da9968237422306598] 
\draw [color={rgb, 255:red, 80; green, 227; blue, 194 }  ,draw opacity=1 ]   (547,85) -- (547,138) ;
%Straight Lines [id:da5742916323360499] 
\draw [color={rgb, 255:red, 0; green, 0; blue, 0 }  ,draw opacity=1 ]   (597,85) -- (597,138) ;
\draw [fill=uuuuuu]   (572,85)  circle (1.5pt);
\draw [fill=uuuuuu]   (597,85)  circle (1.5pt);
\draw [fill=uuuuuu]   (547,85)  circle (1.5pt);
%
%Straight Lines [id:da19001143688046596] 
\draw    (572,182) -- (597,154) ;
\draw [fill=uuuuuu]   (572,182)  circle (1.5pt);
%Straight Lines [id:da7582373081147207] 
\draw [color={rgb, 255:red, 189; green, 16; blue, 224 }  ,draw opacity=1 ]   (572,182) -- (572,154) ;
%Straight Lines [id:da8982963670459347] 
\draw [color={rgb, 255:red, 80; green, 227; blue, 194 }  ,draw opacity=1 ]   (572,182) -- (547,154) ;
%Straight Lines [id:da26357151983808746] 
\draw [color={rgb, 255:red, 74; green, 144; blue, 226 }  ,draw opacity=1 ]   (572,182) -- (525,154) ;
%%%%%%%%%%%%%%%%%%%%%%%%%%%%%%%%%%%%%%%%%%%%%%%%
\draw [fill=uuuuuu]   (85,138)  circle (1.5pt);
\draw [fill=uuuuuu]   (55,85)  circle (1.5pt);
\draw [fill=uuuuuu]   (85,85)  circle (1.5pt);
\draw [fill=uuuuuu]   (115,85)  circle (1.5pt);
%Straight Lines [id:da8380813526262467] 
\draw [color={rgb, 255:red, 0; green, 0; blue, 0 }  ,draw opacity=1 ]   (55,85) -- (85,138) ;
%Straight Lines [id:da3101708773028098] 
\draw [color={rgb, 255:red, 0; green, 0; blue, 0 }  ,draw opacity=1 ]   (85,85) -- (85,138) ;
%Straight Lines [id:da6532079273962004] 
\draw [color={rgb, 255:red, 0; green, 0; blue, 0 }  ,draw opacity=1 ]   (115,85) -- (85,138) ;
\draw [fill=uuuuuu]   (163,138)  circle (1.5pt);
\draw [fill=uuuuuu]   (149,85)  circle (1.5pt);
\draw [fill=uuuuuu]   (179,85)  circle (1.5pt);
\draw [fill=uuuuuu]   (209,85)  circle (1.5pt);
%Straight Lines [id:da95884372161758] 
\draw [color={rgb, 255:red, 0; green, 0; blue, 0 }  ,draw opacity=1 ]   (179,85) -- (163,138) ;
%Straight Lines [id:da32569897229997014] 
\draw [color={rgb, 255:red, 0; green, 0; blue, 0 }  ,draw opacity=1 ]   (115,85) -- (163,138) ;
%Straight Lines [id:da6197464730069933] 
\draw [color={rgb, 255:red, 0; green, 0; blue, 0 }  ,draw opacity=1 ]   (149,85) -- (163,138) ;
%Straight Lines [id:da16446391168803665] 
\draw [color={rgb, 255:red, 0; green, 0; blue, 0 }  ,draw opacity=1 ]   (209,85) -- (163,138) ;
%Curve Lines [id:da6354211410433341] 
\draw [color={rgb, 255:red, 0; green, 0; blue, 0 }  ,draw opacity=1 ]   (55,85) .. controls (75,60) and (95,60) .. (115,85) ;
%Curve Lines [id:da15195394101074577] 
\draw [color={rgb, 255:red, 0; green, 0; blue, 0 }  ,draw opacity=1 ]   (149,85) .. controls (169,60) and (189,60) .. (209,85) ;
%%%%%%%%%%%%%%%%%%%%%%%%%%%%%%%%%%%%%%%%%%%%%%%%%%
\draw [fill=uuuuuu]   (240,85)  circle (1.5pt);
\draw [fill=uuuuuu]   (270,85)  circle (1.5pt);
\draw [fill=uuuuuu]   (300,85)  circle (1.5pt);
%Straight Lines [id:da17682940183646645] 
\draw [color={rgb, 255:red, 0; green, 0; blue, 0 }  ,draw opacity=1 ]   (240,85) -- (240,138) ;
%Straight Lines [id:da12696660977471486] 
\draw [color={rgb, 255:red, 0; green, 0; blue, 0 }  ,draw opacity=1 ]   (270,85) -- (270,138) ;
%Straight Lines [id:da7493455047102673] 
\draw [color={rgb, 255:red, 0; green, 0; blue, 0 }  ,draw opacity=1 ]   (300,85) -- (300,138) ;
%Curve Lines [id:da5189263993572166] 
\draw [color={rgb, 255:red, 0; green, 0; blue, 0 }  ,draw opacity=1 ]   (240,85) .. controls (260,60) and (280,60) ..  (300,85) ;
\draw [fill=uuuuuu]   (370,85)  circle (1.5pt);
\draw [fill=uuuuuu]   (340,85)  circle (1.5pt);
\draw [fill=uuuuuu]   (400,85)  circle (1.5pt);
%Straight Lines [id:da5925045055261562] 
\draw [color={rgb, 255:red, 0; green, 0; blue, 0 }  ,draw opacity=1 ]   (370,85) -- (370,138) ;
%Straight Lines [id:da09776499191558963] 
\draw [color={rgb, 255:red, 0; green, 0; blue, 0 }  ,draw opacity=1 ]   (300,85) -- (325,138) ;
%Straight Lines [id:da7514258990785672] 
\draw [color={rgb, 255:red, 0; green, 0; blue, 0 }  ,draw opacity=1 ]   (340,85) -- (340,138) ;
%Straight Lines [id:da473194809106537] 
\draw [color={rgb, 255:red, 0; green, 0; blue, 0 }  ,draw opacity=1 ]   (400,85) -- (400,138) ;
%Curve Lines [id:da436720213788258] 
\draw [color={rgb, 255:red, 0; green, 0; blue, 0 }  ,draw opacity=1 ]   (340,85) .. controls (360,60) and (380,60) .. (400,85) ;

% Text Node
\draw (459,193) node [anchor=north west][inner sep=0.75pt]   [align=left] {$S_1$};
% Text Node
\draw (563,193) node [anchor=north west][inner sep=0.75pt]   [align=left] {$S_2$};
% Text Node
\draw (510,230) node [anchor=north west][inner sep=0.75pt]   [align=left] {$F^4$};
% Text Node
\draw (459,33) node [anchor=north west][inner sep=0.75pt]   [align=left] {$e_8$};
% Text Node
\draw (563,33) node [anchor=north west][inner sep=0.75pt]   [align=left] {$e_9$};
% Text Node
\draw (79,147) node [anchor=north west][inner sep=0.75pt]   [align=left] {$v_1$};
% Text Node
\draw (158,147) node [anchor=north west][inner sep=0.75pt]   [align=left] {$v_2$};
% Text Node
\draw (118,230) node [anchor=north west][inner sep=0.75pt]   [align=left] {$F$};
% Text Node
\draw (230,144) node [anchor=north west][inner sep=0.75pt]   [align=left] {$v_1^{1}$};
% Text Node
\draw (260,144) node [anchor=north west][inner sep=0.75pt]   [align=left] {$v_{1}^2$};
% Text Node
\draw (290,144) node [anchor=north west][inner sep=0.75pt]   [align=left] {$v_{1}^3$};
% Text Node
\draw (313,144) node [anchor=north west][inner sep=0.75pt]   [align=left] {$v_2^1$};
% Text Node
\draw (335,144) node [anchor=north west][inner sep=0.75pt]   [align=left] {$v_2^2$};
% Text Node
\draw (360,144) node [anchor=north west][inner sep=0.75pt]   [align=left] {$v_2^3$};
% Text Node
\draw (390,144) node [anchor=north west][inner sep=0.75pt]   [align=left] {$v_2^4$};
%
% Text Node
\draw (280,230) node [anchor=north west][inner sep=0.75pt]   [align=left] {$F\vee\{v_1, v_2\}$};
%Straight Lines [id:da1335354180715178] 
\draw  [dashed]  (222,30) -- (222,250) ;
\draw  [dashed]  (422,30) -- (422,250) ;
\end{tikzpicture}
\caption{From $F$ to $F\vee\{v_1, v_2\}$ and then to  $F^4$, where pairs with the same color form a hyperedge.}
\label{fig:splitting}
\end{figure}
%uncomment if require: \path (0,254); %set diagram left start at 0, and has height of 254
%%%%%%%%%%%%%%%%%%%%%%%%%%%%%%%%%%%%%%
%%%%%%%%%%%%%%%%%%%%%%%%%%%%%%%%

\begin{proof}[Proof of Theorem~\ref{THM:antiRamsey-expansion-hypergraph-splitting}]
    Let $r > k \ge 2$ be integers and $F$ be a $k$-graph. 
    Let $n$ be sufficiently large.
    Let $\chi\colon K_{n}^{r} \to \mathbb{N}$ be a rainbow-$H_{F}^{r}$-free coloring and $\mathcal{H} \subseteq K_{n}^{r}$ be a rainbow subgraph. 
    Let $\mathcal{C}$ be a maximal collection of pairwise edge-disjoint copies of members in $H_{F_{-}}^{r}$. In other words, members in $\mathcal{C}$ are pairwise edge-disjoint, and if $H \subseteq \mathcal{H}$ is a copy of some member in $H_{F_{-}}^{r}$, then $H$ contains at least one edge from some member in $\mathcal{C}$. 
    For convenience, let us assume that $\mathcal{C} = \{Q_1, \ldots, Q_m\}$, where $m \ge 0$ is an integer and $Q_i \subseteq \mathcal{H}$ is a copy of some member in $H_{F_{-}}^{r}$ for $i\in [m]$. 

    Let $i\in [m]$ and $S\subseteq [n]\setminus V(Q_i)$ be an $(r-k)$-set.  Since $Q_i$ is a copy of some member in $H_{F_{-}}^{r}$, there exists a $k$-set $e_i \subseteq V(Q_i)$ such that $Q_i \cup \{\{e_i \cup S\}\}$ is a copy of $H_{F}^{r}$. 
    We let $\mathcal{A}$ be an auxiliary multi-$k$-graph whose edge set is the collection of $e_i$ for all $i\in [m]$. 
    For $i\in [m]$ let $\chi(Q_i) := \{\chi(e)\colon e\in Q_i\}$. 
    Since $\mathcal{H}$ is rainbow and $Q_i \subseteq \mathcal{H}$ for $i\in [m]$, we know that 
    \begin{align}\label{equ:anti-Ramsey-splitting-disjoint-color-set}
        \chi(Q_i) \cap  \chi(Q_j) = \emptyset
        \quad\text{for all}\quad 1\le i < j \le m. 
    \end{align}

    \begin{claim}\label{CLAIM:anti-Ramsey-splitting-e_i-cup-S}
        For every $i\in [m]$ and for every $(r-k)$-set $S\subseteq [n]\setminus V(Q_i)$, we have $\chi(e_i \cup S) \in \chi(Q_i)$. 
    \end{claim}
    \begin{proof}
        Suppose to the contrary that $\chi(e_i \cup S) \not\in \chi(Q_i)$, then the $r$-graph $Q_i \cup \{\{e_i \cup S\}\}$ would be a rainbow copy of $H_{F}^{r}$, a contradiction. 
    \end{proof}

    \begin{claim}\label{CLAIM:anti-Ramsey-splitting-no-multiset}
        The set $\mathcal{A}$ does not contain multi-sets. In other words, $e_i \neq e_j$ for all $1 \le i < j \le m$. 
    \end{claim}
    \begin{proof}
        Suppose to the contrary that $e_i = e_j =: e$ for some $i\neq j$. 
        Let $S\subseteq  [n]\setminus \left(V(Q_i)\cup V(Q_j)\right)$ be an $(r-k)$-set. 
        It follows from Claim~\ref{CLAIM:anti-Ramsey-splitting-e_i-cup-S} and~\eqref{equ:anti-Ramsey-splitting-disjoint-color-set} that $\chi(e \cup S) \in  \chi(Q_i) \cap  \chi(Q_j) = \emptyset$, a contradiction. 
    \end{proof}

    \begin{claim}\label{CLAIM:antiRamsey-splitting-free}
        The $k$-graph $\mathcal{A}$ is $\mathrm{Split}(F)$-free. 
        In particular, $m \le \mathrm{ex}(n, \mathrm{Split}(F))$.     
    \end{claim}
    \begin{proof}
        Suppose to the contrary that $\mathcal{A}$ contains some member $\hat{F} \in \mathrm{Split}(F)$. 
        By the definition of $\mathrm{Split}(F)$, there exists an independent set $I= \{v_1, \ldots, v_{p}\}$ in $F$ such that $\hat{F} = F\vee I$, where $p\ge 0$ is an integer.
        Let $d_i:= d_{F}(v_i)$ for $i\in [p]$, and let $d:= \sum_{i\in[p]}d_i$. 
        Assume that $\hat{F} = \{f_1, \ldots, f_{\ell}\}$, where $\ell:= |\hat{F}|$.
        Let $\psi\colon \hat{F} \to \mathcal{A}$ be an embedding. 
        By relabelling members in $\mathcal{C}$ if necessary, we may assume that $\psi(f_i) = e_i$ for $i\in [\ell]$. 
        Let $U:= \bigcup_{i\in [\ell]}V(Q_i)$.
        \begin{itemize}
            \item  For $i\in [p]$, choose a $d_i$-star\footnote{A $d$-star is a collection of $d$ edges such that there is a unique vertex (called center) that is contained in all edges and two edges intersect only on this vertex.} $S_{d_i}$ from $\binom{[n]\setminus U}{r-k}$, and 
            \item for $j \in \left[d+1, \ell\right]$, choose an $(r-k)$-subset $e'_j$ of $[n]\setminus U$
        \end{itemize}
        such that elements in $\left\{S_{d_1}, \ldots, S_{d_p}, e'_{d+1}, \ldots, e'_{\ell}\right\}$ are pairwise vertex-disjoint. 
       We will use $\{e_1, \ldots, e_{\ell}\}$ and $\left\{S_{d_1}, \ldots, S_{d_p}, e'_{d+1}, \ldots, e'_{\ell}\right\}$ to build a rainbow copy of $H_{F}^{r}$. 

       By relabelling members in $\hat{F}$ if necessary, we may assume that
        \begin{align*}
            \{f_1, \ldots, f_{d_1}\}, \ldots, \{f_{d_1+\cdots+d_{p-1}+1}, \ldots, f_{d}\} \subseteq \hat{F}
        \end{align*}
        are edge sets obtained by splitting $v_1, \ldots, v_p$, respectively. 
       In other words,
       \begin{align*}
           \{f_{d_1+\cdots+d_{i-1}+1},\ldots,f_{d_1+\cdots+d_i}\} = (F\vee v_{i})\setminus F
           \quad\text{for}\quad i\in [p]. 
       \end{align*}
        We further assume that $S_{d_i} = \{e'_{d_1+\cdots+d_{i-1}+1},\ldots,e'_{d_1+\cdots+d_i}\}$ for $i\in [p]$. 
        Now let $E_j := e_j \cup e'_j$ for $j\in [\ell]$. 
        It is easy to observe that $\{E_1, \ldots, E_{\ell}\}$ is a copy of $H_{F}^{r}$ with the center of $S_{d_i}$ playing the role of $v_i$ for $i\in [p]$ (see Figure~\ref{fig:splitting} for an example when $k=2$ and $r=4$). 
        In addition, it follows from Claim~\ref{CLAIM:anti-Ramsey-splitting-e_i-cup-S} and~\eqref{equ:anti-Ramsey-splitting-disjoint-color-set} that $\{E_1, \ldots, E_{\ell}\}$ is rainbow, which contradicts the rainbow-$H_{F}^{r}$-freeness of the coloring $\chi$. 
    \end{proof}%CLAIM
    Let $\mathcal{H}':= \mathcal{H}\setminus\left(\bigcup_{i\in [m]}Q_i\right)$. 
    It follows from the maximality of $\mathcal{C}$ that $\mathcal{H}'$ is $H_{F_{-}}^{r}$-free.  
    In addition, it follows from Claim~\ref{CLAIM:antiRamsey-splitting-free} that $|\mathcal{H}'| \ge |\mathcal{H}| - |F|\cdot m \ge |\mathcal{H}|-|F|\cdot \mathrm{ex}(n, \mathrm{Split}(F))$, completing the proof of Theorem~\ref{THM:antiRamsey-expansion-hypergraph-splitting}. 
\end{proof}%LEMMA
%
%%%%%%%%%%%%%%%%%%%%%%%%%%%%%%%%%%%%%%%%%%%%%%%%%%%%
\section{Proofs of Theorems~\ref{THM:antiRamsey-expansion-edge-criticalr=3} and~\ref{THM:antiRamsey-expansion-edge-criticalr>3}}\label{SEC:proof-antiRamsey-expansion-edge-criticalr=3}
We prove Theorems~\ref{THM:antiRamsey-expansion-edge-criticalr=3} and~\ref{THM:antiRamsey-expansion-edge-criticalr>3} in this section. Before that, we prove some useful lemmas. 
\begin{lemma}\label{LEMMA:ESS-removal}
    Let $r \ge 2$, $F$ be an $r$-graph, and $\chi\colon K_{n}^{r} \to \mathbb{N}$ be a rainbow-$F$-free coloring. 
    Every rainbow subgraph $\mathcal{H} \subseteq K_{n}^{r}$ can be made $F_{-}$-free by removing $o(n^r)$ edges.  
\end{lemma}
\begin{proof}
    The lemma follows easily from the Hypergraph Removal Lemma (see e.g.~\cite{NRS06,RS06,Tao06,Gow07}) and the observation of Erd{\H o}s--Simonovits--S\'{o}s~\cite{ESS75} that every rainbow subgraph $\mathcal{H} \subseteq K_{n}^{r}$ is $F_{-}[2]$-free. 
\end{proof}
We also need the following strengthen of~{\cite[Lemma~4.5]{LMR1}}. 

Let $\mathcal{G}$ be an $r$-graph with vertex set~$[m]$ and let $V_1 \sqcup \cdots \sqcup V_{m} = V$ be a partition of some vertex set $V$. 
We use $\mathcal{G}(V_1,\ldots,V_{m})$ to denote the $r$-graph obtained by replacing vertex $i$ in $\mathcal{G}$ with the set $V_i$ for $i\in [m]$, and by replacing each edge in $\mathcal{G}$ with a corresponding complete $r$-partite $r$-graph. We call $\mathcal{G}(V_1,\ldots,V_{m})$ a \textbf{blowup} of $\mathcal{G}$. 
\begin{lemma}\label{LEMMA:greedily-embedding-Gi}
Fix a real $\eta \in (0, 1)$ and integers $m, n\ge 1$.
Let $\mathcal{G}$ be an $r$-graph with vertex set~$[m]$ and let $\mathcal{H}$ be a further $r$-graph
with $v(\mathcal{H})=n$.
Consider a vertex partition $V(\mathcal{H}) = \bigcup_{i\in[m]}V_i$ and the associated
blow-up $\widehat{\mathcal{G}} := \mathcal{G}(V_1,\ldots,V_{m})$ of $\mathcal{G}$.
Suppose that two sets $T \subseteq [m]$ and $S\subseteq V(\mathcal{H})$
have the properties
\begin{enumerate}[label=(\alph*)]
\item\label{it:47a} $|V_{j}'| \ge 2q(|S|+1)|T|\eta^{1/r} n$  for all $j \in T$,
\item\label{it:47b} $|\mathcal{H}[V_{j_1}',\ldots,V_{j_r}']| \ge |\widehat{\mathcal{G}}[V_{j_1}',\ldots,V_{j_r}']|
		- \eta n^r$ for all $\{j_1,\ldots,j_r\} \in \binom{T}{r}$, and
\item\label{it:47c} $|L_{\mathcal{H}}(v)[V_{j_1}',\ldots,V_{j_{r-1}}']| \ge |L_{\widehat{\mathcal{G}}}(v)[V_{j_1}',\ldots,V_{j_{r-1}}']|
		- \eta n^{r-1}$ for all $v\in S$ and for all $\{j_1,\ldots,j_{r-1}\} \in \binom{T}{r-1}$,
\end{enumerate}
where $V_i' := V_i\setminus S$ for $i\in [m]$. 
Then there exists a selection of $q$-set $U_j \subseteq V_j$ for all $j\in [T]$
such that $U := \bigcup_{j\in T}U_j$ satisfies
$\widehat{\mathcal{G}}[U] \subseteq \mathcal{H}[U]$ and
$L_{\widehat{\mathcal{G}}}(v)[U] \subseteq L_{\mathcal{H}}(v)[U]$ for all $v\in S$.
In particular, if $\mathcal{H} \subseteq \widehat{\mathcal{G}}$,
then $\widehat{\mathcal{G}}[U] = \mathcal{H}[U]$ and
$L_{\widehat{\mathcal{G}}}(v)[U] = L_{\mathcal{H}}(v)[U]$ for all $v\in S$.
\end{lemma}
\begin{proof}   
    For $i\ge 2$ we use $K_{q,\ldots,q}^{i}$ to denote the complete $i$-partite $i$-graph with each part of size $q$. 
    We use the notation $N(F,G)$ to represent the number of copies of $F$ in $G$. 
    By shrinking $V_j'$ if necessary, we may assume that $|V_j'| = n_1:= q(|S|+1)|T|\eta^{1/r} n$. 
    Choose for each $j\in T$ a $q$-set $U_{j} \subseteq V_j'$ independently and uniformly at random. 
    Let $U:= \bigcup_{j\in T}U_j$. 
    For every $\{j_1, \ldots, j_{r}\} \in \mathcal{G}$, 
    let $\mathbb{P}_{j_1, \ldots, j_{r}}$ denote the probablity that $\mathcal{H}[U_{j_1}, \ldots, U_{j_r}]\neq\widehat{\mathcal{G}}[U_{j_1}, \ldots, U_{j_r}]$. 
    Then it follows from Assumption~\ref{it:47b} that 
    \begin{align*}
        % \mathbb{P}\left(\mathcal{H}[U_{j_1}, \ldots, U_{j_r}]\neq\widehat{\mathcal{G}}[U_{j_1}, \ldots, U_{j_r}]\right)
        \mathbb{P}_{j_1, \ldots, j_{r}}
         = 1- \frac{N\left(K_{q,\ldots,q}^{r}, \mathcal{H}[V'_{j_1}, \ldots, V'_{j_r}]\right)}{N\left(K_{q,\ldots,q}^{r}, \widehat{\mathcal{G}}[V'_{j_1}, \ldots, V'_{j_r}]\right)} 
        & \le {\eta n^r \left\{\binom{n_1-1}{q-1}\right\}^r}/{\left\{\binom{n_1}{q}\right\}^r} \\
        %& \le \eta n^r \left(\frac{q}{n_1}\right)^r 
        & = \eta n^r \left(\frac{q}{2q(|S|+1)|T|\eta^{1/r} n}\right)^r
         \le \frac{1}{2|T|^r}. 
    \end{align*}
    Here the numerator comes from the fact that each non-edge of $\mathcal{H}[V'_{j_1}, \ldots, V'_{j_r}]$ is contained in $\left\{\binom{n_1-1}{q-1}\right\}^r$ copies of $K_{q,\ldots,q}^{r}$ in $\widehat{\mathcal{G}}[V'_{j_1}, \ldots, V'_{j_r}]$. 
    
    For every $v\in S$ and $\{j_1, \ldots, j_{r-1}\} \in L_{\mathcal{G}}(v)$, let $\mathbb{P}_{v; j_1, \ldots, j_{r-1}}$ denote the probablity that $L_{\mathcal{H}}(v)[U_{j_1}, \ldots, U_{j_{r-1}}]\neq L_{\widehat{\mathcal{G}}}(v)[U_{j_1}, \ldots, U_{j_{r-1}}]$. 
    Then it follows from Assumption~\ref{it:47c} that 
    \begin{align*}
        % \mathbb{P}\left(L_{\mathcal{H}}(v)[U_{j_1}, \ldots, U_{j_{r-1}}]\neq L_{\widehat{\mathcal{G}}}(v)[U_{j_1}, \ldots, U_{j_{r-1}}]\right)
        \mathbb{P}_{v; j_1, \ldots, j_{r-1}}
         = 1- \frac{N\left(K_{q,\ldots,q}^{r-1}, L_{\mathcal{H}}[V'_{j_1}, \ldots, V'_{j_{r-1}}]\right)}{N\left(K_{q,\ldots,q}^{r-1}, L_{\widehat{\mathcal{G}}}[V'_{j_1}, \ldots, V'_{j_{r-1}}]\right)} 
        & \le {\eta n^{r-1} \left\{\binom{n_1-1}{q-1}\right\}^{r-1}}/{\left\{\binom{n_1}{q}\right\}^{r-1}} \\
        % & \le \eta n^{r-1} \left(\frac{q}{n_1}\right)^{r-1} 
        & = \eta n^{r-1} \left(\frac{q}{2q(|S|+1)|T|\eta^{1/r} n}\right)^{r-1} \\
        & \le \frac{1}{2(|S|+1)|T|^{r-1}}. 
    \end{align*}
    Therefore, the probability that $U$ fails to have the desired properties is at most 
    \begin{align*}
        \binom{|T|}{r}\times \frac{1}{2|T|^r}
        + |S|\binom{|T|}{r-1} \times  \frac{1}{2(|S|+1)|T|^{r-1}}
        \le \frac{1}{4} + \frac{1}{2}
        = \frac{3}{4}. 
    \end{align*}
    So the probability that $U$ has these properties is positive. 
\end{proof}
Another ingredient that we need is the following stability result for $H_{\ell+1}^{r}[t]$-free $r$-graphs. 
\begin{lemma}\label{LEMMA:stability-expansion-clique}
    Let $\ell \ge r \ge 3$ and $t \ge 1$ be fixed integers. 
    For every $\epsilon>0$ there exist $\delta >0$ and $n_0$ such that the following holds for all $n \ge n_0$. 
    Suppose that $\mathcal{H}$ is a $H_{\ell+1}^{r}[t]$-free $r$-graph with $n \ge n_0$ vertices and at least $t_{r}(n,\ell) - \delta n^r$ edges. 
    Then $\mathcal{H}$ can be made $\ell$-partite by removing at most $\epsilon n^r$ edges.  
\end{lemma}
\begin{proof}
    This lemma follows easily from the Hypergraph Removal Lemma and the stability theorem of Mubayi on $H_{\ell+1}^{r}$-free $r$-graphs~{\cite[Theorem~3]{M06}}. 
\end{proof}

Now we are ready to prove Theorem~\ref{THM:antiRamsey-expansion-edge-criticalr=3}. 
\begin{proof}[Proof of Theorem~\ref{THM:antiRamsey-expansion-edge-criticalr=3}]
Fix integers $t > \ell \ge 3$. 
The lower bound for the case $F = \ K_{\ell}^{\gamma}[t]$ follows the well-known fact
\begin{align*}
    \mathrm{ar}(n,F) \ge \mathrm{ex}(n,F_{-}) +2 
    \quad\text{holds for all $n \ge 1$ and for all $r$-graphs $F$}.
\end{align*}
The lower bound for the case $F \in \left\{K_{\ell}^{\alpha}[t], K_{\ell}^{\beta}[t] \right\}$ follows from the following construction.  

Fix $n \ge \ell \ge 3$. Let $m := t_{3}(n,\ell)+\ell$. 
Let $V_1 \sqcup \cdots \sqcup V_{\ell} = [n]$ be a partition such that $|V_{\ell}|+1 \ge |V_1| \ge |V_2| \ge \cdots \ge |V_{\ell}|$. 
\begin{itemize}
    \item For $i\in [\ell]$ color all triples that contain at least two vertices from $V_i$ using color $i$, and  
    \item fix an arbitrary rainbow coloring for the rest $t_{3}(n,\ell)$ triples using colors that were not used in the previous step. 
\end{itemize}
%
%We leave the verification of the rainbow-$H_{F}^{3}$-freeness of the coloring defined above to interested readers. 
We use an argument similar to the proof of~{\cite[Claim~22.2]{LP22}} to show that there is no rainbow copy of $H_{F}^{3}$ for $F \in \left\{K_{\ell}^{\alpha}[t], K_{\ell}^{\beta}[t] \right\}$ under the coloring defined above (denoted by $\chi$). 
Suppose to the contrary that there exists an embedding $\psi \colon V(H_{F}^{3}) \to [n]$ of $H_{F}^{3}$ to $K_{n}^{3}$ such that $\psi(H_{F}^{3})$ is rainbow under $\chi$. 
For simplicity, let us assume that $F = K_{\ell}^{\alpha}[t]$; the proof for the case $K_{\ell}^{\beta}[t]$ is similar. 
Let $U_1 \cup \cdots \cup U_{\ell} = V(F)$ be the partition such that $F[U_i]$ is empty for $i \in [2,\ell]$, and $F[U_1]$ consists of a path of length two and isolated vertices (see Figure~\ref{fig:edge-critical}~(a)).  
By the Pigeonhole Principle, for each $i \in [\ell]$, we can fix an index $f(i) \in [\ell]$ such that 
\begin{align*}
    |U_i \cap V_{f(i)}| 
    \ge \lceil t/\ell \rceil
    \ge 2.
\end{align*}
\begin{claim}\label{CLAIM:lower-bound-a}
    For every $i \in [\ell]$, we have 
    \begin{align*}
        V_{f(i)} \cap \psi(U_j) 
        = \emptyset 
        \quad\text{for every}~j \in [\ell] \setminus \{i\}.
    \end{align*}
\end{claim}
\begin{proof}
    Suppose to the contrary that there exist distinct $i,j \in [\ell]$ such that $V_{f(i)} \cap \psi(U_j) \neq \emptyset$.
    By symmetry, we may assume that $(i,j) = (1,2)$ and $f(1) = 1$.
    Fix two distinct vertices $u_1, u_2 \in U_1$ such that $\psi(u_1), \psi(u_2) \in V_1$, and fix a vertex $u \in U_2$ such that $\psi(u) \in V_{1} \cap \psi(U_2)$. 
    Let $e_1$ and $e_2$ be two edges in $H_{F}^{3}$ containing $\{u_1, u\}$ and $\{u_2, u\}$, respectively. 
    Since $\psi(H_{F}^{3})$ is rainbow under the coloring $\chi$, we have $\chi(\psi(e_1)) \neq \chi(\psi(e_2))$. 
    However, since $\min\{|\psi(e_1) \cap V_1|, |\psi(e_2) \cap V_1|\} \ge 2$, by the definition of $\chi$, we have $\chi(\psi(e_1)) = \chi(\psi(e_2)) = 1$, a contradiction. 
\end{proof}

It follows from Claim~\ref{CLAIM:lower-bound-a} that the map $f\colon [\ell] \to [\ell]$ is a bijection, and furthermore, $\psi(U_i) \subseteq  V_{f(i)}$ for every $i \in [\ell]$. 

Now, let $\{w_1w_2, w_2w_3\} \subseteq  F[U_1]$ denote the path of length two.
Let $\hat{e}_1$ and $\hat{e}_2$ be two edges in $H_{F}^{3}$ containing $\{w_1, w_2\}$ and $\{w_2, w_3\}$, respectively.
Similar to the proof of Claim~\ref{CLAIM:lower-bound-a}, since $\min\{|\psi(\hat{e}_1) \cap V_{f(1)}|, |\psi(\hat{e}_2) \cap V_{f(1)}|\} \ge 2$, we have $\chi(\psi(\hat{e}_1)) = \chi(\psi(\hat{e}_2)) = f(1)$, which means that $\psi(H_{F}^{3})$ cannot be rainbow under the coloring $\chi$.
Therefore, $\mathrm{ar}(n,H_{F}^{3}) \ge t_{3}(n,\ell)+\ell+1$.

Now we focus on the upper bound. 
Fix $F\in \left\{K_{\ell}^{\alpha}[t],\ K_{\ell}^{\beta}[t],\ K_{\ell}^{\gamma}[t]\right\}$. 
Let $\delta >0$ be sufficiently small and $n$ be a sufficiently large. 
Let $\chi\colon K_{n}^{3} \to \mathbb{N}$ be a rainbow $H_{F}^{3}$-free coloring, and let $\mathcal{H} \subseteq K_{n}^{3}$ be a maximum rainbow subgraph. 
Suppose to the contrary that 
\begin{align*}
        |\mathcal{H}|
            \ge 
            \begin{cases}
                t_{3}(n,\ell) + \ell+1, & \quad\text{if}\quad F \in \left\{K_{\ell}^{\alpha}[t],\ K_{\ell}^{\beta}[t]\right\}, \\
                t_{3}(n,\ell) + 2, & \quad\text{if}\quad F = K_{\ell}^{\gamma}[t].
            \end{cases}
\end{align*}
Let $F_1:= K_{\ell}^{+}[t]$, where $K_{\ell}^{+}[t]$ denote the graph obtained from $K_{\ell}[t]$ by adding one edge into some part.
Note that $F_1 \in F_{-}$. 
Let $\mathcal{H}' \subseteq \mathcal{H}$ be a maximum $H_{F_1}^{3}$-free subgraph. It follows from Lemma~\ref{LEMMA:ESS-removal} that 
\begin{align}\label{equ:edge-critical-H'r=3}
    |\mathcal{H}'| 
    \ge |\mathcal{H}| -  \frac{\delta n^3}{100}  
    % \ge t_{\ell}(n,3) - o(n^3)
    \ge \binom{\ell}{3}\left(\frac{n}{\ell}\right)^3 - \frac{\delta n^3}{50}. 
\end{align}
Let $V_1 \sqcup \cdots \sqcup V_{\ell} = [n]$ be partition such that the number of edges in the induced $\ell$-partite subgraph\footnote{Here, $\mathcal{H}[V_1, \ldots, V_{\ell}]$ consists of all edges in $\mathcal{H}$ that contain at most one vertex from each $V_i$.} $\mathcal{H}'' := \mathcal{H}[V_1, \ldots, V_{\ell}]$ is maximized. 
Since $H_{F_1}^3$ is contained in $H_{\ell+1}^{3}[t]$, it follows from~\eqref{equ:edge-critical-H'r=3} and  Lemma~\ref{LEMMA:stability-expansion-clique} that 
\begin{align}\label{equ:edge-critical-H''r=3}
    |\mathcal{H}''| 
    \ge |\mathcal{H}'| - \frac{\delta n^3}{25}
    \ge \binom{\ell}{3}\left(\frac{n}{\ell}\right)^3 - \frac{\delta n^3}{18}. 
\end{align}
Combined with the inequality (see~{\cite[Lemma~2.2]{LMR2}})
    \begin{align*}
        |\mathcal{H}''|
        \le \sum_{1\le i < j < k \le \ell}|V_i||V_j||V_k|
        \le \binom{\ell}{3}\left(\frac{n}{\ell}\right)^3 - \frac{\ell-2}{6\ell}\sum_{i\in [\ell]}\left(|V_i|- \frac{n}{\ell}\right)^2n, 
    \end{align*}
    we obtain 
    \begin{align}\label{equ:edge-critical-Vi-sizer=3}
        \frac{n}{\ell} - \sqrt{\delta}n 
        \le |V_i|
        \le \frac{n}{\ell} + \sqrt{\delta}n
        \quad\text{for all}\quad i\in [\ell]. 
    \end{align}
Let $\mathcal{G}$ denote the complete $r$-graph with $\ell$ vertices, and let $\widehat{\mathcal{G}}:= \mathcal{G}(V_1, \ldots, V_{\ell})$ denote the complete $\ell$-partite $r$-graph\footnote{In other words, $\widehat{\mathcal{G}}$ consists of all $r$-subset of $[n]$ that contain at most one vertex from each $V_i$.} with parts $V_1, \ldots, V_{\ell}$. 
Let 
\begin{align*}
    \mathcal{B} := \mathcal{H} \setminus \widehat{\mathcal{G}},
    \quad\text{and}\quad
    \mathcal{M} := \widehat{\mathcal{G}} \setminus \mathcal{H} = \widehat{\mathcal{G}} \setminus \mathcal{H}''. 
\end{align*}
It follows from~\eqref{equ:edge-critical-H''r=3} that 
\begin{align}\label{equ:edge-critical-Mr=3}
    |\mathcal{M}|
    = |\widehat{\mathcal{G}}| - |\mathcal{H}''|
    \le \binom{\ell}{3}\left(\frac{n}{\ell}\right)^3 - \left(\binom{\ell}{3}\left(\frac{n}{\ell}\right)^3 - \frac{\delta n^3}{18}\right)
    \le \delta n^3. 
\end{align}
For $i\in [\ell]$ let 
\begin{align*}
    D_i := \left\{v\in V_i \colon |L_{\widehat{\mathcal{G}}}(v)| - |L_{\mathcal{H}''}(v)| \le 3\delta^{1/3} n^2\right\}, 
    \quad\text{and}\quad
    \overline{D}_i := V_i \setminus D_i. 
\end{align*}
For convenience, let $D:= \bigcup_{i\in [\ell]} D_i$ and
$\overline{D} := \bigcup_{i\in [\ell]} \overline{D}_i$. 
\begin{claim}\label{CLAIM:edge-critical-missing-edgesr=3}
    We have $|\mathcal{M}| \ge \delta^{1/3} n^2 |\overline{D}|$ and $|\overline{D}| \le \delta^{2/3} n$. 
\end{claim}
\begin{proof}
    It follows from the definition that every vertex in $\overline{D}$ contributes at least $3\delta^{1/3} n^2$ elements in $\mathcal{M}$. 
    Therefore,  
    \begin{align*}
        |\mathcal{M}| 
        \ge \frac{1}{3} \times |\overline{D}| \times 3\delta^{1/3} n^2
        = \delta^{1/3} n^2 |\overline{D}|.  
    \end{align*}
    Combined with~\eqref{equ:edge-critical-Mr=3}, we obtain $|\overline{D}| \le \delta n^3/(\delta^{1/3} n^2) = \delta^{2/3} n$, 
    which completes the proof of Claim~\ref{CLAIM:edge-critical-missing-edgesr=3}. 
\end{proof}
The most crucial part in the proof is the following claim. 
\begin{claim}\label{CLAIM:edge-critical-two-bad-edgesr=3}
    If $F\in \left\{K_{\ell}^{\alpha}[t],\ K_{\ell}^{\beta}[t]\right\}$. Then $\mathcal{B}$ does not contain two edges $e, e'$ such that 
    \begin{align}\label{equ:CLAIM:edge-critical-two-bad-edgesr=3-case-1}
           \min\left\{|e\cap D_i|,\ |e'\cap D_i|\right\} \ge 2
           \quad\text{holds for some }i\in [\ell]. 
       \end{align}
    If $F = K_{\ell}^{\gamma}[t]$, then 
    the set $\mathcal{B}$ does not contain two edges $e, e'$ such that 
       \begin{align}\label{equ:CLAIM:edge-critical-two-bad-edgesr=3}
           \max\left\{|e\cap D_i|\colon i\in [\ell]\right\} \ge 2
           \quad\text{and}\quad
           \max\left\{|e'\cap D_i| \colon i\in [\ell]\right\} \ge 2. 
       \end{align}
\end{claim}
\begin{proof}
    First consider the case $F = K_{\ell}^{\alpha}[t]$. 
    Suppose to the contrary that there exist two distinct edges $e, e' \in \mathcal{B}$ such that~\eqref{equ:CLAIM:edge-critical-two-bad-edgesr=3-case-1} holds. 
    By symmetry, we may assume that $\min\left\{|e\cap D_1|,\ |e'\cap D_1|\right\} \ge 2$. 
    Choose a $3$-set $f \subseteq D_1$ such that $|f\cap e| = |f\cap e'| = 1$. 
    By symmetry, we may assume that $\chi(f) \neq \chi(e)$. 
    Let $\{v_1\}:= f\cap e$. Fix $v_2 \in \left(e\cap D_1\right)\setminus \{v_1\}$ and $v_3 \in f\setminus\{v_1\}$. 
    Let $\mathcal{H}'$ be the $3$-graph obtained from $\mathcal{H}$ by removing an edge (if there exists such an edge) with color $\chi(f)$. 
    Then apply Lemma~\ref{LEMMA:greedily-embedding-Gi} to $\mathcal{H}'$ with $V_i' = D_i\setminus (e \cup f)$ for $i\in [\ell]$, $T = [\ell]$, $S = \{v_1, v_2, v_3\}$, and $q = 2\binom{\ell}{2}t^2$, we obtain a $q$-set $U_i\subseteq D_i\setminus (e \cup f)$ for each $i\in [\ell]$ such that 
    the induced $\ell$-partite $3$-graph of $\mathcal{H}'$ on $U_1, \ldots, U_{\ell}$ is complete, and for every $v\in \{v_1, v_2, v_3\}$ the induced $(\ell-1)$-partite graph of $L_{\mathcal{H}'}(v)$ on $U_2, \ldots, U_{\ell}$ is also complete. 
    Note that Lemma~\ref{LEMMA:greedily-embedding-Gi}~\ref{it:47a} is guaranteed by~\eqref{equ:edge-critical-Vi-sizer=3}, Lemma~\ref{LEMMA:greedily-embedding-Gi}~\ref{it:47b} is guaranteed by~\eqref{equ:edge-critical-Mr=3}, and Lemma~\ref{LEMMA:greedily-embedding-Gi}~\ref{it:47c} is guaranteed by the definition of $D_i$. 
    Let $U:= \{v_1, v_2, v_3\} \cup \bigcup_{i\in [\ell]}U_i$
    It is easy to see that the expansion of $K_{\ell}^{\alpha}[t]$ is contained in $\{e,f\} \cup \mathcal{H}'[U]$.  This is a contradiction, since $\{e,f\} \cup \mathcal{H}'[U]$ is rainbow. 

    The case $F = K_{\ell}^{\beta}[t]$ can be proved similarly by choosing a $3$-set $f \subseteq D_1$ such that $|f\cap e| = |f\cap e'| = 0$. So we omit the details here.

    Now we consider the case $F = K_{\ell}^{\gamma}[t]$. 
    Suppose to the contrary that there exist two distinct edges $e, e' \in \mathcal{B}$ such that~\eqref{equ:CLAIM:edge-critical-two-bad-edgesr=3} holds. 
    Let us assume that $|e\cap D_{i_1}|\ge 2$ and $|e'\cap D_{i_2}|\ge 2$, where $i_1, i_2 \in [\ell]$. 
    Fix $\{u_1,v_1\} \subseteq  e\cap D_{i_1}$ and $\{u_1',v_1'\} \subseteq  e'\cap D_{i_2}$. 
    Then apply Lemma~\ref{LEMMA:greedily-embedding-Gi} to $\mathcal{H}$ with $V_i' = D_i\setminus (e \cup f)$ for $i\in [\ell]$, $T = [\ell]$, $S = \{u_1, v_1, u_1', v_1'\}$, and $q = 2\binom{\ell}{2}t^2$, we obtain a $q$-set $U_i\subseteq D_i\setminus (e \cup f)$ for each $i\in [\ell]$ such that 
    \begin{itemize}
        \item the induced $\ell$-partite $3$-graph of $\mathcal{H}$ on $U_1, \ldots, U_{\ell}$ is complete,
        \item for every $v\in \{u_1, v_1\}$ the induced $(\ell-1)$-partite graph of $L_{\mathcal{H}}(v)$ on $\{U_i \colon i\in [\ell]\setminus\{i_1\}\}$ is complete,
        \item and for every $v\in \{u_1', v_1'\}$ the induced $(\ell-1)$-partite graph of $L_{\mathcal{H}}(v)$ on $\{U_i \colon i\in [\ell]\setminus\{i_2\}\}$ is complete. 
    \end{itemize}
    Choose $i^{\ast} \in [\ell]\setminus\{i_1,i_2\}$ and fix a $3$-set $f \subseteq D_{i^{\ast}}$ such that $|f\cap U_{i^{\ast}}| = 2$. 
    By symmetry, we may assume that $\chi(f) \neq \chi(e)$. 
    Let $\{u_2, v_2\} := f\cap U_{i^{\ast}}$.
    Fix a $t$-set $W_i \subseteq U_i$ for $i \in [\ell]\setminus\{i_1, i^{\ast}\}$,  
    fix a $t$-set $W_{i_1} \subseteq U_{i_1} \cup \{u_1, v_1\}$ with $\{u_1, v_1\} \subseteq W_{i_1}$, and fix a $t$-set $W_{i^{\ast}} \subseteq U_{i^{\ast}}$ with $\{u_2, v_2\} \subseteq W_{i^{\ast}}$. 
    Let $K$ denote the complete $\ell$-partite graph with parts $W_1, \ldots, W_{\ell}$. 
    Observe from the choice of $U_i$'s that every pair in $K$ is contained in at least $q = 2\binom{\ell}{2}t^2$ edges in $\mathcal{H}$. 
    Since $\mathcal{H}$ is rainbow, it is easy to greedily extend $K$ to be a rainbow copy of $H_{K}^{3}$ and avoid using the color $\chi(f)$. 
    However, this copy of $H_{K}^{3}$ together with edges $e$ and $f$ is a rainbow copy of $H_{F}^{3}$, contradicting the rainbow-$H_{F}^{3}$-freeness of $\chi$. 
\end{proof}
For $i\in \{0,1,2,3\}$ let 
    \begin{align*}
        \mathcal{B}_i := \left\{e\in \mathcal{B} \colon |e\cap \overline{D}| = i\right\}. 
    \end{align*}

\medskip 

\textbf{Case 1}: $F\in \left\{K_{\ell}^{\alpha}[t],\ K_{\ell}^{\beta}[t]\right\}$.

    For every triple $e\in \mathcal{B}_0 \cup \mathcal{B}_1$ we can fix a pair $e'\subseteq e\cap D_i$ for some $i\in [\ell]$. 
    It follows from Claim~\ref{CLAIM:edge-critical-two-bad-edgesr=3} that no two triples will share the same pair and no two pairs lie in the same part. 
    Therefore, $|\mathcal{B}_0| + |\mathcal{B}_1| \le \ell$.
    Combined with Claim~\ref{CLAIM:edge-critical-missing-edgesr=3} and the trivial bound $|\mathcal{B}_2| + |\mathcal{B}_3| \le n|\overline{D}|^2$, we obtain 
    \begin{align*}
        |\mathcal{H}| 
        \le |\widehat{\mathcal{G}}| + \sum_{i=0}^{3}|\mathcal{B}_i| - |\mathcal{M}|
        & \le t_{3}(n,\ell) + \ell + n|\overline{D}|^2 - \delta^{1/3} n^2 |\overline{D}| \\
        & \le t_{3}(n,\ell) + \ell - \left(\delta^{2/3} n - \delta^{1/3} n\right)n|\overline{D}|
        \le t_{3}(n,\ell) + \ell, 
    \end{align*}
    a contradiction. 
    % Note that equality holds only if $\overline{D} = \emptyset$ and $\mathcal{M} = \emptyset$. 

\medskip 

\textbf{Case 2}: $F = K_{\ell}^{\gamma}[t]$. 

    Similarly, for every triple $e\in \mathcal{B}_0 \cup \mathcal{B}_1$ we can fix a pair $e'\subseteq e\cap D_i$ for some $i\in [\ell]$. 
    It follows from Claim~\ref{CLAIM:edge-critical-two-bad-edgesr=3} that no two triples will share the same pair and the number of pairs is at most one. 
    Therefore, $|\mathcal{B}_0| + |\mathcal{B}_1| \le 1$.
    Similarly, by Claim~\ref{CLAIM:edge-critical-missing-edgesr=3} and the trivial bound $|\mathcal{B}_2| + |\mathcal{B}_3| \le n|\overline{D}|^2$, we have 
    \begin{align*}
        |\mathcal{H}| 
        \le |\widehat{\mathcal{G}}| + \sum_{i=0}^{3}|\mathcal{B}_i| - |\mathcal{M}|
        & \le t_{3}(n,\ell) + 1 + n|\overline{D}|^2 - \delta^{1/3} n^2 |\overline{D}| \\
        & \le t_{3}(n,\ell) + 1 - \left(\delta^{2/3} n - \delta^{1/3} n\right)n|\overline{D}|
        \le t_{3}(n,\ell) + 1, 
    \end{align*}
    a contradiction. 
    % Note that equality holds only if $\overline{D} = \emptyset$ and $\mathcal{M} = \emptyset$. 
\end{proof}%THM

%%%%%%%%%%%%%%%%%%%%%%%%%%%%%%%%%%%%
% \subsection{Proof of Theorem~\ref{THM:antiRamsey-expansion-edge-criticalr>3}}\label{SUBSEC:proof-THM:antiRamsey-expansion-edge-criticalr>3}
% %
% \begin{proof}[Proof of Theorem~\ref{THM:antiRamsey-expansion-edge-criticalr>3}]
%     
The proof for Theorem~\ref{THM:antiRamsey-expansion-edge-criticalr>3} is similar to the proof of Theorem~\ref{THM:antiRamsey-expansion-edge-criticalr=3}. So we omit the detail and only sketch the proof for the most crucial claim. 
\begin{claim}\label{CLAIM:edge-critical-two-bad-edgesr>3}
    The set $\mathcal{B}$ does not contain two edges $e, e'$ such that 
       \begin{align}\label{equ:CLAIM:edge-critical-two-bad-edgesr>3}
           \max\left\{|e\cap D_i|\colon i\in [\ell]\right\} \ge 2
           \quad\text{and}\quad
           \max\left\{|e'\cap D_i| \colon i\in [\ell]\right\} \ge 2. 
       \end{align}
\end{claim}
\begin{proof}
    Suppose to the contrary that there exist two distinct edges $e, e' \in \mathcal{B}$ such that~\eqref{equ:CLAIM:edge-critical-two-bad-edgesr>3} holds. 
    Let us assume that $|e\cap D_{i_1}|\ge 2$ and $|e'\cap D_{i_2}|\ge 2$, where $i_1, i_2 \in [\ell]$.

    If $F = K_{\ell}^{\alpha}[t]$, then we choose an $r$-set $f \subseteq D_{i_1} \cup D_{i_2}$ such that $|f\cap e| = |f\cap e'| = 1$ and such that $\min\left\{|f\cap D_{i_1}|,\ |f\cap D_{i_2}|\right\} \ge 2$. Since $r\ge 4$, such $f$ exists. 

    If $F = K_{\ell}^{\beta}[t]$, then we choose an $r$-set $f \subseteq D_{i_1} \cup D_{i_2}$ such that $|f\cap e| = |f\cap e'| = 0$ and such that $\min\left\{|f\cap D_{i_1}|,\ |f\cap D_{i_2}|\right\} \ge 2$. Since $r\ge 4$, such $f$ exists. 

    The case $F = K_{\ell}^{\gamma}[t]$ can be handled using the same way as in the proof of Claim~\ref{CLAIM:edge-critical-two-bad-edgesr=3}. 

    The rest part is similar to the proof of Claim~\ref{CLAIM:edge-critical-two-bad-edgesr=3}, so we omit the details. 
\end{proof}
%\end{proof}%THM

\section{Concluding remarks}\label{SEC:remark}
$\bullet$ Given an $r$-graph $F$ and an integer $1\le k < r$, we say an edge $e\in F$ is \textbf{$k$-pendant} if $e$ contains a $k$-subset $e'$ such that 
\begin{align*}
    e'\cap f = \emptyset
    \quad\text{for all}\quad f\in F\setminus \{e\}.
\end{align*}
For convenience, let $F_{k-}$ denote the family of $r$-graphs that can be obtained from $F$ by removing one $k$-pendant edge, i.e. 
\begin{align*}
    F_{k-} := \left\{F\setminus \{e\} \colon \text{$e\in F$ is $k$-pendant}\right\}. 
\end{align*}
The argument in the proof of Theorem~\ref{THM:antiRamsey-expansion-hypergraph-splitting} (see Claim~\ref{CLAIM:anti-Ramsey-splitting-no-multiset}) yields the following result. 
\begin{theorem}\label{THM:anti-Ramsey-pendant-edge}
    Let $r > k \ge 1$ be integers and $F$ be an $r$-graph. 
    Suppose that $\chi\colon K_{n}^{r} \to \mathbb{N}$ is a rainbow-$F$-free coloring. 
    Then every rainbow subgraph $\mathcal{H} \subseteq K_{n}^{r}$ can be made $F_{k-}$-free by removing at most $(|F|-1)\binom{n}{k}$ edges.
    In particular, for all integers $n \ge r$, 
    \begin{align*}
        \mathrm{ar}(n,F) \le \mathrm{ex}(n,F_{k-}) + (|F|-1) \binom{n}{k}. 
    \end{align*}
\end{theorem}

%%%%%%%%
$\bullet$ The following question seems interesting, but we are uncertain whether it has already been addressed in the literature. 

\begin{problem}
    Let $r \ge 2$ be an integer. 
    Is it true that for every $\delta >0$ there exists an $r$-graph $F$ such that 
    \begin{align}\label{equ:antiRamsey-prob}
        \mathrm{ar}(n,F) - \mathrm{ex}(n,F_{-}) = \Omega(n^{r-\delta})? 
    \end{align}
    Furthermore, does there exists an $r$-graph $F$ such that~\eqref{equ:antiRamsey-prob} holds for all $\delta >0$? 
    % \begin{align*}
    %     \mathrm{ar}(n,F) - \mathrm{ex}(n,F_{-}) = \Omega(n^{r-\delta})? 
    % \end{align*}
\end{problem}
%%%%%%%%%%%%%%%%%%%%%%%%%%%%%%%%%%%%%%%%%%%%%%%%
\section*{Acknowledgement}
We are very grateful to the referees for their helpful comments.
We also thank Yucong Tang for bringing~\cite{LTY24a} to our attention.
%%%%%%%%%%%%%%%%%%%%%%%%%%%%%%%%%%%%%%%%%%%%%%%%%%
\bibliographystyle{abbrv}
\bibliography{GeneralizedTuran}
%%%%%%%%%%%%%%%%%%%%%%%%%%%%%%%%%%%%%%%%%%%%%%%
%%%%%%%%%%%%%%%%%%%%%%%%%%%%%%%%%%%%%%%%%%%%%%%%%%
\end{document}